\documentclass[10pt,reqno]{amsart}
\usepackage{graphicx}
\usepackage{amsfonts}
\usepackage{amssymb}
\usepackage{amsmath}
\usepackage{amsxtra}
\usepackage{latexsym}
\usepackage{epstopdf}
\usepackage{mathrsfs}
\usepackage{esint}
\usepackage{subfigure}
\usepackage{enumitem}

\newtheorem{theorem}{Theorem}[section]
\newtheorem{lemma}[theorem]{Lemma}

\newtheorem{corollary}[theorem]{Corollary}
\newtheorem{proposition}[theorem]{Proposition}

\theoremstyle{definition}
\newtheorem{definition}[theorem]{Definition}
\newtheorem{algorithm}{Algorithm}
\newtheorem{example}[theorem]{Example}

\theoremstyle{remark}
\newtheorem{remark}[theorem]{Remark}

\numberwithin{equation}{section}

\begin{document}

\title[]
{}

\title[On growth error bound conditions]
{On growth error bound conditions with an application to heavy ball method}

\author{Qinian Jin}
\address{Mathematical Sciences Institute, Australian National
University, Canberra, ACT 2601, Australia}
\email{qinian.jin@anu.edu.au} \curraddr{}




\begin{abstract}
In this paper, we investigate the growth error bound condition. By using the proximal point algorithm, 
we first provide a more accessible and elementary proof of the fact that Kurdyka-{\L}ojasiewicz 
conditions imply growth error bound conditions for convex functions which has been established 
before via a subgradient flow. We then extend the result for nonconvex functions. Furthermore we show 
that every definable function in an o-minimal structure must satisfy a growth error bound condition.
Finally, as an application, we consider the heavy ball method for solving convex optimization problems 
and propose an adaptive strategy for selecting the momentum coefficient. Under growth error bound 
conditions, we derive convergence rates of the proposed method. A numerical experiment is conducted 
to demonstrate its acceleration effect over the gradient method. 
\end{abstract}


\def\d{\delta}
\def\E{\mathbb{E}}
\def\P{\mathbb{P}}
\def\X{\mathcal{X}}
\def\A{\mathcal{A}}
\def\Y{\mathcal{Y}}
\def\l{\langle}
\def\r{\rangle}
\def\by{\overline{y}^{(n)}}
\def\bY{\overline{y}_n}
\def\la{\lambda}
\def\EE{{\mathbb E}}
\def\RR{{\mathbb R}}
\def\a{\alpha}
\def\l{\langle}
\def\r{\rangle}
\def\p{\partial}
\def\ep{\varepsilon}
\def\O{\mathcal O}
\def\M{\mathcal M}
\def\prox{\mbox{prox}}

%
%


\maketitle

\section{\bf Introduction}

Many algorithms have been developed to solve minimization problems of the form 
\begin{align}\label{mp}
\min_{x\in X} f(x),
\end{align}
where $X$ is a Hilbert space and $f: X \to (-\infty, \infty]$ is a proper lower semi-continuous function. 
Various favorable error bound conditions have been proposed on $f$ to guarantee desired convergence rates 
for many specific algorithms (\cite{AB2009,ABS2013,DL2018,LT1993,NNG2019}). These error bound conditions 
constitute important mathematical concepts with diverse applications in optimization, analysis, and 
variational inequalities. They play a crucial role in establishing convergence properties of algorithms, 
characterizing critical points, and understanding the behavior of functions in various settings.

Let $S:= \arg\min_{x\in X} f(x)$ denote the set of solutions of (\ref{mp}) which is assumed to be non-empty. 
One of the prominent error bound conditions is the growth error bound condition. A proper lower 
semi-continuous function $f: X \to (-\infty, \infty]$ is said to satisfy a growth error bound 
condition at a point $\bar x \in S$ if 
\begin{align}\label{GEB}
d(x, S) \le \varphi(f(x) - f(\bar x)), \quad \forall x \in B_r(\bar x) \cap \{f< f(\bar x) + \eta\}
\end{align}
for some $r>0$, $\eta\in (0, \infty]$ and a desingularizing function $\varphi \in {\mathcal K}(0, \eta)$, 
where $B_r(\bar x)$ denotes the open ball of radius $r$ with center $\bar x$,
$$
{\mathcal K}(0, \eta) 
:= \{\varphi \in C[0, \eta) \cap C^1(0, \eta): \varphi(0) =0, \varphi \mbox{ is concave},  
\varphi'(s)>0\},
$$
and $d(x, S)$ denotes the distance from $x$ to $S$, i.e. 
$$
d(x, S) := \inf_{z\in S} \|x - z\|.
$$
When $\varphi(t) = \sqrt{\gamma t}$ for some constant $\gamma >0$, (\ref{GEB}) becomes 
$$
d^2(x, S) \le \gamma(f(x) - f(\bar x)), \quad \forall x\in B_r(\bar x)\cap \{f< f(\bar x) + \eta\}
$$
which is known as the quadratic growth condition. 

The equivalence of the growth error bound condition to other error bound conditions has been explored for 
convex functions, see \cite{AG2014,BNPS2017,DL2018} for instance. In particular, the relation between 
the growth error bound condition and the Kurdyka-{\L}ojasiewicz condition has been investigated in 
\cite{BNPS2017}. For a proper, lower semi-continuous, convex function $f: X \to (-\infty, \infty]$ satisfying 
the growth error bound condition (\ref{GEB}), it has been proved in \cite{BNPS2017} that if there exists 
$c>0$ such that 
\begin{align}\label{phi}
s \varphi'(s) \ge c \varphi(s) \quad  \mbox{for all } s \in (0, \eta), 
\end{align}
then $f$ must satisfy at $\bar x$ the condition
\begin{align}\label{KL0}
\varphi'(f(x) - f(\bar x)) d(0, \p f(x)) \ge 1, \quad 
\forall x \in B_r(\bar x) \cap \{f(\bar x) < f < f(\bar x) + \eta\}
\end{align}
and, moreover, the converse holds without requiring (\ref{phi}), where $\p f$ denotes the convex 
subdifferential of $f$. The condition (\ref{KL0}) is known as the Kurdyka-{\L}ojasiewicz (KL) condition 
(\cite{BDL2006,BDLS2007,K1998,L1963,L1965})
which plays a remarkable role in recent development on first order optimization methods. When $f$ is 
continuously differetiable, the KL condition with $\varphi(t) = \sqrt{\gamma t}$ for some constant 
$\gamma >0$ is also called the Polyak-{\L}ojasiewicz condition (\cite{L1963,L1965,P1963}). 

The proof that (\ref{GEB}) implies (\ref{KL0}) is somewhat straightforward. However, establishing the reverse 
implication is highly nontrivial. The proof provided in \cite{BNPS2017} is based on the subgradient flow
\begin{align*}
\left\{\begin{array}{lll}
\dot{x}(t) \in  - \p f(x(t)) \quad \mbox{for almost all } t \in (0, \infty), \\
x(0) = x 
\end{array} \right.
\end{align*}
generated by $\p f$ and hence the argument relies heavily on the theory of differential inclusion which 
involves a lot of machinery. In this paper, by using the proximal point algorithm we will first provide an 
alternative argument to show that the KL condition (\ref{KL0}) implies the growth error bound condition 
(\ref{GEB}) for proper, lower semi-continuous, convex functions, offering a more accessible and elementary 
proof. Furthermore, we will extend our argument to demonstrate that, for proper lower semi-continuous 
nonconvex functions, (\ref{KL0}) still implies (\ref{GEB}) on a possibly smaller set. As a direct 
consequence, we may use the result in \cite{BDLS2007}, which states that every definable function in an 
o-minimal structure must satisfy the KL condition, to conclude that every definable function satisfies the 
growth error bound condition (\ref{GEB}). Nevertheless, we will provide a direct derivation of this fact 
because the proof of KL property of definable functions in \cite{BDLS2007} involves a lot of machinery, 
including the Whitney stratification and the projection formula related to the stratum. 

As an application, we will consider the heavy ball method for solving convex optimization problems. This 
method can be viewed as a modification of the gradient method by adding a momentum term. The choice of the 
momentum coefficient crucially affects the performance of the method. We will develop an adaptive strategy 
for selecting the momentum coefficient and, under growth error bound conditions, we will derive the 
convergence rates of our proposed method. We will also conduct a numerical experiment related to the computed 
tomography which demonstrates that our method has obvious acceleration effect over the gradient method and it 
can produce computational results comparable to those obtained by Nesterov's accelerated gradient method and 
the ALR-HB method, i.e. the heavy ball method with adaptive learning rate in \cite{WJZ2023}, with even better 
accuracy as the iterations proceed.

\section{\bf KL condition implies growth error bound condition}

We start by proving that the KL condition implies the growth error bound condition for proper lower 
semi-continuous convex functions. 

\begin{theorem}\label{thm:KLGC}
Let $f: X \to (-\infty, \infty]$ be a proper, lower semi-continuous, convex function 
with $S:=\arg\min_{x\in X} f(x) \ne \emptyset$. Let $\bar x \in S$ and $f_*:= f(\bar x)$. 
If there exist $r>0$, $\eta>0$ and a desingularizing function 
$\varphi\in {\mathcal K}(0, \eta)$ such that 
$$
\varphi'(f(x) - f_*)  d(0, \p f(x)) \ge 1
$$
for all $x \in B_r(\bar x) \cap \{f_* < f < f_* + \eta\}$, then 
$$
d(x, S) \le 2 \varphi(f(x) - f_*)
$$
for all $x \in B_r(\bar x)\cap \{f_* \le f < f_* + \eta\}$. 
\end{theorem}

\begin{proof}
Given any $x \in B_r(\bar x)\cap \{f_* < f< f_* + \eta\}$ and consider the sequence $\{x_k\}$ defined by the 
proximal point algorithm 
$$
x_k = \arg\min_{z\in X} \left\{ f(z) + \frac{1}{2\tau} \|z- x_{k-1}\|^2\right\}, 
\quad k = 1, 2, \cdots
$$
with $x_0 : = x$, where $\tau>0$ is a constant. Because $f$ is convex and lower 
semi-continuous, we have the following results (see \cite[Theorem 28.1]{BC2011} 
for instance): 

\begin{enumerate}[leftmargin = 0.9cm]
\item[$\bullet$] $\|x_{k+1} - \bar x\| \le \|x_k - \bar x\|$ for all $k \ge 0$. 
This implies $x_k \in B_r(\bar x)$ for all $k\ge 0$. 

\item[$\bullet$] $f(x_{k+1}) \le f(x_k)$ for all $k$. Consequently 
$x_k \in \{f_* \le f < f_* + \eta\}$ for all $k \ge 0$. 

\item[$\bullet$] $f(x_k) \to f_*$ and $x_k \rightharpoonup x^*$ as $k \to \infty$ for 
some $x^* \in S$. where $\rightharpoonup$ denotes the weak convergence in $X$. 
\end{enumerate}
We now show that for all integers $k \ge 1$ there holds
\begin{align}\label{KLGC.1}
2 \|x_{k+1} - x_k\| \le \|x_k - x_{k-1}\| + 2 \left[\varphi(f(x_k) - f_*) - \varphi(f(x_{k+1}) - f_*)\right].
\end{align}
If $f(x_k) = f_*$, then by the definition of $x_{k+1}$ we have 
\begin{align}\label{KLGC.2}
f(x_{k+1}) + \frac{1}{2\tau} \|x_{k+1} - x_k\|^2 \le f(x_k) = f_*
\end{align}
which implies $f(x_{k+1}) = f_*$ and $x_{k+1} = x_k$ and hence (\ref{KLGC.1}) holds trivially. So we may assume 
$f(x_k) > f_*$. By the definition of $x_k$ we have $(x_{k-1}-x_k)/\tau \in \p f(x_k)$. Thus, by the given condition 
we have  
\begin{align}\label{KLGC.3}
\varphi'(f(x_k)-f_*) \ge \frac{1}{d(0, \p f(x_k))} \ge \frac{\tau}{\|x_k - x_{k-1}\|}. 
\end{align}
It follows from the concavity of $\varphi$, (\ref{KLGC.2}) and (\ref{KLGC.3}) that 
\begin{align*}
\varphi(f(x_{k+1})-f_*) 
& \le \varphi(f(x_k)-f_*) + \varphi'(f(x_k)-f_*) (f(x_{k+1}) - f(x_k)) \\
& \le \varphi(f(x_k) - f_*) - \frac{\|x_{k+1} - x_k\|^2}{2 \|x_k - x_{k-1}\|}.
\end{align*}
Therefore 
\begin{align*}
2 \|x_{k+1} - x_k\| 
& \le \|x_k - x_{k-1}\| + \frac{\|x_{k+1} - x_k\|^2}{\|x_k - x_{k-1}\|} \\
& \le \|x_k - x_{k-1}\| + 2 \left[\varphi(f(x_k)-f_*) - \varphi(f(x_{k+1}) - f_*)\right]
\end{align*}
which shows (\ref{KLGC.1}) again. 

For any integer $l>1$, by summing (\ref{KLGC.1}) over $k$ from $1$ to $l-1$ we can obtain 
\begin{align*}
2 \sum_{k=1}^{l-1} \|x_{k+1} - x_k\| 
& \le \sum_{k=1}^{l-1} \|x_k - x_{k-1}\| + 2 \left[\varphi(f(x_1)-f_*) - \varphi(f(x_{l}) - f_*)\right] \\
& \le \sum_{k=1}^{l-1} \|x_k-x_{k-1}\| + 2\varphi(f(x_{1})-f_*) \\
& \le \sum_{k=1}^{l-1} \|x_k - x_{k-1}\| + 2 \varphi(f(x_0)-f_*).
\end{align*}
Consequently 
\begin{align}\label{KLGC.11}
\|x_{l} - x_0\| \le \sum_{k=0}^{l-1} \|x_{k+1} - x_k\| \le 2 \|x_1 - x_0\| +  2 \varphi(f(x_0)-f_*).
\end{align}
Since $x_l \rightharpoonup x^*$ as $l \to \infty$, we have 
$$
\|x^*-x\|^2 = \lim_{l\to \infty} \l x_l-x, x^*-x\r 
\le \liminf_{l\to \infty} \|x_l-x\| \|x^*-x\|
$$
which implies $\|x^*-x\| \le \liminf_{l\to \infty} \|x_l - x\|$. Therefore, by using $x^* \in S$, 
$x_0 = x$ and (\ref{KLGC.11}) we can obtain
\begin{align*}
d(x, S) \le \|x^* - x\| \le \liminf_{l\to \infty} \|x_l - x_0\| \le 2 \|x_1 - x_0\| + 2 \varphi(f(x) - f_*). 
\end{align*}
By the definition of $x_1$ we have $\|x_1 - x_0\|^2 \le 2\tau (f(x_0) - f(x_1)) 
\le 2 \tau (f(x) - f_*)$. Therefore 
\begin{align*}
d(x, S) \le 2 \sqrt{2\tau (f(x) - f_*)}  + 2 \varphi(f(x) - f_*)
\end{align*}
which is valid for all $\tau>0$. Letting $\tau \to 0$ then shows the result.  
\end{proof}

For a given convex function, we may use Theorem \ref{thm:KLGC} to derive the growth error bound of its 
Moreau envelop from the growth error bound condition of the function itself. Recall that for a proper, lower 
semi-continuous, convex function $f: X \to (-\infty, \infty]$, its Moreau envelop $M_{\la f}$, for any 
$\la>0$, is defined by 
$$
M_{\la f}(x) := \inf_{x\in X} \left\{f(z) + \frac{1}{2\la} \|z-x\|^2\right\}, \quad x \in X.
$$
It is well known (\cite{BC2011}) that $M_{\la f}$ is a continuous differentiable convex 
function with $M_{\la f}(x) \le f(x)$ for all $x \in X$,
$$
\min_{x\in X} f(x) = \min_{x\in X} M_{\la f}(x),
$$
and $f$ and $M_{\la f}$ have the same set of minimizers; moreover 
\begin{align}\label{ME}
\nabla M_{\la f}(x) = \frac{1}{\la} \left(x- \prox_{\la f}(x)\right), \quad x\in X,
\end{align}
where $\prox_{\la f}$ denotes the proximal mapping of $f$ with parameter $\la$, i.e. 
$$
\prox_{\la f}(x) := \arg\min_{x\in X} \left\{f(z) + \frac{1}{2\la} \|z-x\|^2\right\}, \quad x \in X. 
$$
By the firmly nonexpansiveness of the proximal mapping, $\nabla M_{\la f}$ is Lipschitz continuous 
with constant $1/\la$. We have the following result. 

\begin{corollary}\label{prop:GEB}
Let $f: X \to (-\infty, \infty]$ be a proper, lower semi-continuous, convex function with 
$S:=\arg\min_{x\in X} f(x) \ne \emptyset$. If $f$ satisfies a growth error bound condition at a point 
$\bar x\in S$, i.e. there exist $C>0$, $r>0$, $\eta \in (0, \infty]$ and $\a \in (0, 1]$ such that 
\begin{align}\label{GEB.1}
d(x, S) \le C \left(f(x) - f(\bar x)\right)^\a
\end{align}
for all $x\in B_r(\bar x) \cap \{f(\bar x) \le f < f(\bar x) + \eta\}$, then there exists a 
constant $\tilde C>0$ such that 
$$
d(x, S) \le \tilde C \left(M_{\la f}(x) - M_{\la f}(\bar x)\right)^{\min\{\a, 1/2\}}
$$
for all $x \in B_r(\bar x) \cap \{M_{\la f}(\bar x) \le M_{\la f} < M_{\la f}(\bar x) + \eta\}$. 
\end{corollary}

\begin{proof}
By using (\ref{GEB.1}) and \cite[Theorem 5 (ii)]{BNPS2017} we first have 
\begin{align}\label{KL.131}
\left(f(x) - f(\bar x)\right)^{1-\a} \le \frac{C}{\a} d(0, \p f(x))
\end{align}
for all $x \in B_r(\bar x) \cap \{f(\bar x) \le f < f(\bar x) + \eta\}$. We next show that 
$M_{\la f}$ satisfies the KL condition at $\bar x$ with exponent $\min\{1-\a, 1/2\}$. 
Let $x \in B_r(\bar x)\cap \{M_{\la f}(\bar x) \le M_{\la f}< M_{\la f}(\bar x)+\eta\}$ 
be any point. By noting that 
$$
M_{\la f}(x) = f(\prox_{\la f}(x)) + \frac{1}{2\la} \|\prox_{\la f}(x) - x\|^2
$$
and (\ref{ME}), we have $f(\bar x) \le f(\prox_{\la f}(x)) \le M_{\la f}(x) < f(\bar x) + \eta$ and 
\begin{align*}
M_{\la f}(x) - M_{\la f}(\bar x) 
= f(\prox_{\la f}(x)) - f(\bar x) + \frac{\la}{2} \|\nabla M_{\la f}(x)\|^2.
\end{align*}
By the nonexpansiveness of the proximal mapping, we have 
$$
\|\prox_{\la f}(x) -\bar x\| = \|\prox_{\la f}(x) - \prox_{\la f}(\bar x)\| \le \|x-\bar x\| <r. 
$$
Thus we may use (\ref{KL.131}) to conclude 
\begin{align*}
M_{\la f}(x) - M_{\la f}(\bar x) 
\le \left(\frac{C}{\a}\right)^{\frac{1}{1-\a}} d(0, \p f(\prox_{\la f}(x)))^{\frac{1}{1-\a}}  
+ \frac{\la}{2} \|\nabla M_{\la f}(x)\|^2.
\end{align*}
By the definition of $\prox_{\la f}(x)$ we have 
$$
\nabla M_{\la f}(x) = \frac{1}{\la} (x-\prox_{\la f}(x)) \in \p f(\prox_{\la f}(x)). 
$$
and thus $d(0, \p f(\prox_{\la f}(x))) \le \|\nabla M_{\la f}(x)\|$. Therefore 
\begin{align*}
M_{\la f}(x) - M_{\la f}(\bar x) 
\le \left(\frac{C}{\a}\right)^{\frac{1}{1-\a}} \|\nabla M_{\la f}(x)\|^{\frac{1}{1-\a}}  
+ \frac{\la}{2} \|\nabla M_{\la f}(x)\|^2
\end{align*}
Note that 
$$
\|\nabla M_{\la f}(x)\| = \|\nabla M_{\la f}(x) - \nabla M_{\la f}(\bar x)\| 
\le \frac{1}{\la} \|x - \bar x\| <\frac{r}{\la}. 
$$
Thus, there must exist a constant $C_1>0$ such that 
\begin{align}\label{KLGC.17}
M_{\la f}(x) - M_{\la f}(\bar x) 
\le C_1 \|\nabla M_{\la f}(x)\|^{\min\{\frac{1}{1-\a}, 2\}}  
\end{align}
for all $x \in B_r(\bar x) \cap \{M_{\la f}(\bar x) \le M_{\la f} \le M_{\la f}(\bar x) + \eta\}$. 
Finally, by using (\ref{KLGC.17}) and Theorem \ref{thm:KLGC} we can complete the proof immediately. 
\end{proof}

\begin{remark}
Let $f:X \to (-\infty, \infty]$ be a proper, lower semi-continuous, convex function that 
is a KL function with exponent $\gamma \in (0, 1]$. It has been proved in \cite[Theorem 3.4]{LP2018}
that if $\gamma \in (0, 2/3)$ then its Moreau envelop $M_{\la f}$ for any $\la>0$ is a KL function 
with exponent $\max\{1/2, \gamma/(2-2\gamma)\}$. In the proof of Corrolary \ref{prop:GEB}, we have 
actually obtained the tighter result that $M_{\la f}$ is a KL function with exponent $\max\{1/2, \gamma\}$
for all $\gamma \in (0, 1]$. This tight result has been proved in \cite{YLP2022} recently as a special case 
of a general result with an involved argument, see Theorem 5.2 and Remark 5.1 in \cite{YLP2022}. 
Our derivation is simple and straightforward. 
\end{remark}

Next we will consider extending Theorem \ref{thm:KLGC} to nonconvex functions. 
We need to replace the convex subdifferential by the Fr\'{e}chet subdifferential. 
Recall that for a proper function $f: X \to (-\infty, \infty]$ its Fr\'{e}chet 
subdifferential at a point $x\in \mbox{dom}(f):=\{x\in X: f(x)<\infty\}$ is the set 
$\hat \p f(x)$ consisting of all vectors $\xi\in X$ satisfying 
$$
\liminf_{y\ne x, y\to x} \frac{f(y) - f(x) - \l \xi, y-x\r}{\|y - x\|} \ge 0.
$$
When $x \notin \mbox{dom}(f)$, we set $\hat{\p} f(x) = \emptyset$. By definition it 
is straightforward to show that
\begin{align}\label{LS.1}
\hat \p (f+g)(x) = \nabla f(x) + \hat \p g(x), \quad \forall x \in X
\end{align}
if $f: X \to \RR$ is continuously differentiable and $g: X \to (-\infty, \infty]$ is 
proper. 

In the following result we will extend Theorem \ref{thm:KLGC} to  nonconvex functions by 
showing that if $f$ satisfies a KL condition at a point $\bar x \in \mbox{dom}(f)$ then 
$f$ satisfies a growth error bound condition at $\bar x$ on a possibly smaller set. 

\begin{theorem}\label{thm2}
Let $f: X \to (-\infty, \infty]$ be a proper, weakly lower semi-continuous function and 
$\bar x \in \emph{dom}(f)$. If there exist $r>0$, $\eta>0$ and a desingularizing 
function $\varphi\in {\mathcal K}(0, \eta)$ such that 
\begin{align}\label{KL}
\varphi'(f(x) - f(\bar x))  d(0, \hat \p f(x)) \ge 1
\end{align}
for all $x \in B_r(\bar x) \cap \{f(\bar x) < f < f(\bar x) + \eta\}$, then there exist
$0<\tilde r\le r$ and $0<\tilde \eta \le \eta$ such that 
$$
d(x, S) \le 2 \varphi(f(x) - f(\bar x))
$$
for all $x \in B_{\tilde r}(\bar x)\cap \{f(\bar x) \le  f < f(\bar x) + \tilde \eta \}$, 
where $S:=\{x \in X: f(x) \le f(\bar x)\}$.
\end{theorem}

\begin{proof}
By considering the function $\tilde f(x) := \max\{f(x) - f(\bar x), 0\}$, which is still proper 
and weakly lower semi-continuous, and noting that 
$$
\bar x \in S = \arg\min_{x\in X} \tilde f(x)  
$$
and $\hat \p \tilde f(x) = \hat \p f(x)$ for $x$ satisfying $f(x) > f(\bar x)$, it suffices to prove the result 
by assuming that $\bar x$ is a minimizer of $f$. In the following we set $f_*:= f(\bar x)$. 

Since $\varphi\in C[0, \eta)$ and $\varphi(0) = 0$, we may take $\tilde r\in (0, r]$ and 
$\tilde \eta\in (0, \eta]$ such that
\begin{align}\label{ch9.1}
\tilde r + 2 \varphi(\tilde \eta) < r.
\end{align}
Let $x\in B_{\tilde r}(\bar x)\cap \{f(\bar x) < f < f(\bar x) + \tilde \eta\}$ be any 
fixed point, we define $x_0 := x$ and 
\begin{align}\label{PPA-NC}
x_k \in  \arg\min_{z\in {\mathbb R}^n} \left\{f(z) + \frac{1}{2\tau} \|z- x_{k-1}\|^2\right\}
\end{align}
for $k = 1, 2, \cdots$, where $\tau>0$ is a constant. Since $S\ne \emptyset$, $f$ is bounded 
from below. Thus, by using the weak lower semi-continuity of $f$, it is easy to show that the 
sequence $\{x_k\}$ is well-defined for $\tau>0$; due to the possible non-convexity of 
$f$, (\ref{PPA-NC}) may have many solutions, we take $x_k$ to be any one of them. 
The sequence $\{x_k\}$ does not possess the nice convergence property in general as in the 
convex situation. However, we may follow \cite{AB2009,ABS2013} and use (\ref{KL}) to show 
that $\{x_k\}$ converges to a point in $S$ with which we can establish the desired 
result. 

Indeed, by the definition of $\{x_k\}$ we have 
\begin{align}\label{H1}
f(x_k) + \frac{1}{2\tau} \|x_k - x_{k-1}\|^2 \le f(x_{k-1})
\end{align}
which implies 
\begin{align}\label{ch9.2}
f_* \le f(x_k) \le f(x_0) = f(x) < f(\bar x) + \tilde\eta 
\quad \mbox{ for all } k\ge 0.
\end{align}
By the optimality of $x_k$ and (\ref{LS.1}) we also have 
$
\frac{1}{\tau} (x_{k-1} - x_k) \in \hat \p f(x_k)
$
and thus 
\begin{align}\label{H2}
d(0, \hat \p f(x_k)) \le \frac{1}{\tau} \|x_k - x_{k-1}\|. 
\end{align}
We will use an induction argument to show that $x_k \in B_r(\bar x)$ for all 
$k\ge 0$ if 
\begin{align}\label{KL.37}
0< \tau < \frac{1}{8 \tilde \eta} \left(r - \tilde r - 2 \varphi(\tilde \eta)\right)^2
\end{align}
which will be assumed in the following. Since $x_0 = x \in B_{\tilde r}(\bar x)$ and 
$\tilde r <r$, this is trivial for $k = 0$. By using (\ref{H1}) with $k=0$ we also have
\begin{align}\label{KL.137}
\|x_{1} - x_0\| \le \sqrt{2\tau(f(x_0)-f(x_{1}))} 
\le \sqrt{2\tau(f(x)-f_*)}.
\end{align}
Therefore, by using (\ref{KL.37}),  
\begin{align*}
\|x_1 - \bar x\| \le \|x_0 - \bar x\| + \sqrt{2\tau(f(x)-f_*)} 
< \tilde r + \sqrt{2\tau \tilde \eta} < r
\end{align*}
which shows $x_1 \in B_r(\bar x)$. Now we assume that $x_k \in B_r(\bar x)$ for all 
$0\le k \le j$ for some $j \ge 1$. We will show that $x_{j+1} \in B_r(\bar x)$. Based on 
(\ref{H1}) and (\ref{H2}) we may use the same argument in the proof of Theorem 
\ref{thm:KLGC} to establish (\ref{KLGC.1}), i.e.
\begin{align*}
2\|x_{k+1}-x_k\| \le \|x_k-x_{k-1} \|
+ 2 \left[\varphi(f(x_k)-f_*) - \varphi(f(x_{k+1})-f_*)\right]
\end{align*}
for all $1\le k\le j$. Summing the above equation over $k$ from $k = 1$ to $k= j$ gives
\begin{align*}
2 \sum_{k=1}^j \|x_{k+1}-x_k\| 
\le \sum_{k=1}^j \|x_k-x_{k-1}\| + 2\varphi(f(x_0)-f_*).
\end{align*}
This implies that
\begin{align}\label{5.7.3}
\sum_{k=1}^j \|x_{k+1}-x_k\| + \|x_{j+1}-x_j\| 
\le \|x_{1} - x_0\| + 2\varphi(f(x_0)-f_*).
\end{align}
Therefore
\begin{align*}
\|x_{j+1}-\bar x\| &\le \sum_{k=1}^j \|x_{k+1}-x_k\| + \|x_{1}-x_0\| + \|x_0-\bar x\| \\
& \le  \|x-\bar x\| + 2 \|x_{1} - x_0\| + 2 \varphi(f(x)-f_*).
\end{align*}
Consequently, by virtue of (\ref{KL.137}), (\ref{ch9.1}) and (\ref{KL.37}) we then obtain
\begin{align*}
\|x_{j+1}-\bar x\| &\le \|x-\bar x\| + 2 \sqrt{2\tau(f(x)-f_*)} + 2 \varphi(f(x)-f_*) \\
& < \tilde r + \sqrt{8\tau \tilde \eta} + 2 \varphi(\tilde\eta)
< r
\end{align*}
which implies that $x_{j+1} \in B_r(\bar x)$.

Since $x_k \in B_r(\bar x)$ for all $k \ge 0$, we can see that (\ref{5.7.3}) 
holds for all $j\ge 0$. In particular
$$
\sum_{k=1}^\infty \|x_{k+1}-x_k\| 
\le \|x_{1}-x_0\| + 2 \varphi(f(x)-f_*) <\infty
$$
which shows that $\{x_k\}$ has a finite length. Therefore $\{x_k\}$ is a Cauchy 
sequence and thus it is convergent, i.e. there is $x^*\in X$ such that 
$x_k\to x^*$ as $k\to \infty$. According to (\ref{H1}) and (\ref{ch9.2}), $\{f(x_k)\}$ 
is monotonically decreasing and thus $\bar f := \lim_{k\to \infty} f(x_k)$ exists 
with $f_* \le \bar f < f_* + \tilde \eta$. If $\bar f \ne f_*$, then 
$f(x_k) \ge \bar f > f_*$ for all $k$. Thus, by (\ref{KL}) and (\ref{H2}) we have 
$$
1 \le \varphi'(f(x_k)- f_*) d(0, \hat \p f(x_k)) 
\le \frac{1}{\tau} \varphi'(f(x_k)-f_*) \|x_k - x_{k-1}\|.
$$
Since $\|x_k - x_{k-1}\| \to 0$, we must have $\varphi'(f(x_k)- f_*) \to +\infty$ 
as $k \to \infty$. By the continuity of $\varphi'$ on $(0, \tilde \eta)$ we then obtain 
$$
+ \infty = \lim_{k\to \infty} \varphi'(f(x_k)-f_*) = \varphi'(\bar f - f_*) <\infty
$$
which is a contradiction. Thus $\bar f = f_*$. By the lower semi-continuity 
of $f$ we have 
$$
f_* \le f(x^*) \le \liminf_{k\to \infty} f(x_k) = \lim_{k\to \infty} f(x_k) =\bar f = f_*. 
$$
Therefore $f(x^*) = f_*$, i.e. $x^*\in S$. Consequently, it follows from (\ref{5.7.3})
and the definition of $x_1$ that  
\begin{align*}
d(x, S) \le \|x^*-x\| &=\lim_{k\to \infty} \|x_k - x_0\| 
\le 2 \|x_1 - x_0\| + 2 \varphi(f(x_0) - f_*)\\
& \le 2 \sqrt{2\tau (f(x)-f_*)} + 2 \varphi(f(x) - f_*).
\end{align*}
Letting $\tau \to 0$ shows the desired estimate. 
\end{proof}

Let us give a brief application of Theorem \ref{thm2} to the Levenberg-Marquardt method 
considered in \cite{AAFV2019} for solving the system of nonlinear equations
$$
h(x) = 0
$$
whose solution set, denoted by $S$, is assumed to be nonempty, where $h: \RR^m \to \RR^n$ is a continuously 
differentiable mapping. This problem is equivalent to solving the minimization problem
$$
\min_{x\in \RR^m} \left\{\psi(x):= \frac{1}{2} \|h(x)\|^2\right\}. 
$$
The Levenberg-Marquardt method takes the form 
$$
x_{k+1} = x_k - (\mu_k I + J_k^T J_k)^{-1} J_k^T h(x_k),
$$
where $J_k$ denote the Jacobian of $h$ at $x_k$. The local convergence property has been analyzed in 
\cite{AAFV2019}. In particular, under the H\"{o}lderian local error bound condition at a solution $x^*\in S$,
i.e. there exist $\beta>0$, $r>0$ and $\d\in (0,1]$ such that 
\begin{align}\label{LEB}
\beta d(x, S) \le \|h(x)\|^\d, \quad \forall x \in B_r(x^*)
\end{align}
and the {\L}ojasiewicz gradient inequality at $x^*$, i.e. there exist $\kappa>0$, 
$\epsilon>0$ and $\theta\in (0, 1)$ such that 
\begin{align}\label{LGI}
\psi(x)^\theta \le \kappa \|\nabla \psi(x)\|, \quad \forall x \in B_\epsilon(x^*),
\end{align}
with careful choices of the regularization parameter $\mu_k>0$ it has been shown in 
\cite[Theorem 1]{AAFV2019} that if $\theta \in (0, 1/2]$ then $\psi(x_k) \to 0$ and 
$d(x_k, S) \to 0$ linearly as $k \to \infty$ and if $\theta\in (1/2, 1)$ then there hold 
the sublinear convergence rates
$$
\psi(x_k) = O(k^{-\frac{1}{2\theta-1}}) \quad \mbox{and} \quad d(x_k, S) = O(k^{-\frac{\d}{2(2\theta-1)}}).
$$
According to Theorem \ref{thm2}, (\ref{LGI}) implies (\ref{LEB}) with $\d = 2(1-\theta)$
and therefore (\ref{LEB}) can be dropped from \cite[Theorem 2]{AAFV2019}.

\section{\bf Growth error bound conditions for definable functions} 

In \cite{BDLS2007} it has been shown that any proper, lower semi-continuous, definable 
function $f: {\mathbb R}^n \to (-\infty, \infty]$ in an o-minimal structure on the 
real field $(\RR, +, \cdot)$ satisfies a KL condition at each point in its domain. 
To be more precise, for any $\bar x \in \mbox{dom}(f)$ and any $r>0$, 
there exist $\eta\in (0, \infty]$ and a definable desingularizing function 
$\varphi \in {\mathcal K}(0, \eta)$ such that
\begin{align} \label{OM:KL0}
\varphi'(f(x) - f(\bar x)) d(0, \p f(x)) \ge 1
\end{align}
for all $x \in B_r(\bar x) \cap \{f(\bar x) < f \le f(\bar x) + \eta\}$, Here $\p f(x)$ 
denotes the limiting subdifferential of $f$ at $x\in {\mathbb R}^n$ which is the set consisting of 
all those vectors $\xi\in {\mathbb R}^n$ such that there exist sequences $\{x_k\}$ and $\{\xi_k\}$ 
in ${\mathbb R}^n$ with the properties
$$
x_k \to x, \quad f(x_k) \to f(x) \quad \mbox{ and } \quad \xi_k\in \hat \p f(x_k)
$$
such that $\xi_k \to \xi$ as $k\to \infty$. Clearly $\hat \p f(x) \subset \p f(x)$ 
and hence from (\ref{OM:KL0}) it follows that  
\begin{align*}
\varphi'(f(x) - f(\bar x)) d(0, \hat \p f(x)) \ge 1
\end{align*}
for all $x \in B_r(\bar x) \cap \{f(\bar x) < f \le f(\bar x) + \eta\}$. Consequently 
we may use Theorem \ref{thm2} to obtain
\begin{align}\label{OM:GEB}
d(x, S) \le 2 \varphi(f(x) - f(\bar x))
\end{align}
for all $x \in B_{\tilde r}(\bar x) \cap \{f(\bar x) \le f < f(\bar x) + \tilde \eta\}$ for some 
$0<\tilde r\le r$ and $0< \tilde \eta \le \eta$, where $S:= \{x\in {\mathbb R}^n: f(x) \le f(\bar x)\}$.

Recall that the proof of (\ref{OM:KL0}) given in \cite{BDLS2007} involves a lot of 
machinery, including the existence of a Whitney stratification for the graph of $f$ and 
a projection formula of $\p f(x)$ onto the tangent space of the stratum containing $x$. 
Thus, the derivation of (\ref{OM:GEB}) using the above procedure can be rather involved. 
In the following we will provide a straightforward derivation of (\ref{OM:GEB}) which 
avoids using those involved machinery. 

Let us first recall the definition of o-minimal structures and definable functions 
(see \cite{C1999,DM1996,K1998}).

\begin{definition}
{\it An o-minimal structure\index{o-minimal structure} on the real field $({\mathbb R}, +, \cdot)$ 
is a sequence $\O = (\O_n)_n$ of collections $\O_n$, $n\in {\mathbb N}$, of subsets of 
${\mathbb R}^n$ satisfying the following axioms:

\begin{enumerate}[leftmargin = 0.8cm]
\item[\emph{(i)}] For every $n \in {\mathbb N}$, the collection $\O_n$ is closed under Boolean operations, i.e.  
finite intersections, unions, and complements.

\item[\emph{(ii)}] If $A \in \O_n$ and $B \in \O_m$, then $A \times B \in \O_{n+m}$.

\item[\emph{(iii)}]  For any set $A \in \O_{n+1}$ there holds $\pi(A) \in \O_n$, where 
$\pi : {\mathbb R}^{n+1} \to {\mathbb R}^n$ denotes the canonical projection onto ${\mathbb R}^n$,

\item[\emph{(iv)}] $\O_n$ contains the family of algebraic subsets of ${\mathbb R}^n$, that is,  
every set of the form $\{x \in {\mathbb R}^n: p(x) = 0\}$, where $p : {\mathbb R}^n \to {\mathbb R}$ 
is a polynomial function.

\item[\emph{(v)}] The elements of $\O_1$ are exactly the finite unions of open intervals and points.
\end{enumerate}
Every set in $\O_n$ is called a definable subset of ${\mathbb R}^n$. 
}
\end{definition}

\begin{definition}
{\it Given an o-minimal structure $\O =(\O_n)_n$ over the real field $({\mathbb R}, +, \cdot)$ and 
a subset $A \subset {\mathbb R}^m$. A mapping $F: A \to {\mathbb R}^n$ is called definable 
in $\O$ if its graph 
$$
\mbox{Gr}(F):=\{(x, y)\in A\times {\mathbb R}^n: y = F(x)\}
$$
is a definable set in ${\mathbb R}^m \times {\mathbb R}^n$, i.e. $\mbox{Gr}(F) \in \O_{m+n}$. 
}
\end{definition}

According to the definition of definable sets and definable functions, it is easy to derive the 
following facts (\cite{K1998}):

\begin{enumerate}[leftmargin = 0.8cm]
\item[$\bullet$] Images and inverse images of definable sets under definable functions are definable.

\item[$\bullet$] The sum and composition of definable functions are definable.

\item[$\bullet$] If $A \subset {\mathbb R}^m$ is a definable set, then $x \to d(x, A)$ is a 
definable function on ${\mathbb R}^m$. More generally, if $g: A \subset {\mathbb R}^m 
\to (-\infty, \infty]$ is a definable function that is bounded from below and $F: A \to {\mathbb R}^n$ 
is a definable mapping, then the function $\varphi: {\mathbb R}^n \to (-\infty, \infty]$ defined by 
$$
\varphi(y) := \left\{\begin{array}{lll}
\inf\{g(x): x \in F^{-1}(y)\} & \mbox{ if } F^{-1}(y) \ne \emptyset,\\
+ \infty & \mbox{ if } F^{-1}(y) = \emptyset
\end{array}\right.
$$
is also definable. 
\end{enumerate}

For single variable definable functions there holds the following monotonicity result, see 
\cite[Lemma 2]{K1998} for instance.  

\begin{lemma}[Monotonicity lemma] \label{OM:lem1}
Let $f: (a, b) \to {\mathbb R}$ be a definable function and $l \in {\mathbb N}$. 
Then there is a finite number of points $a = t_0 < t_1 < \cdots< t_k < t_{k+1} = b$ such that, 
on every interval $(t_i, t_{i+1})$, $f$ is $C^l$ and either strictly monotone or constant.
\end{lemma}

Next we prove a result which generalizes the well known {\L}ojasiewicz inequality for continuous 
definable functions on compact sets, see \cite[Theorem 0]{K1998} for instance. Our result 
allows one of the functions to be lower semi-continuous which is crucial for proving 
(\ref{OM:GEB}). We will make use of ideas from \cite{LP2022}. 

\begin{proposition}\label{SA.lem3}
Let $K$ be a compact definable set in ${\mathbb R}^n$ and let $g: K \to [0, \infty)$ 
and $h: K \to [0, \infty]$ be definable functions. If $g$ is continuous, $h$ is 
lower semi-continuous and 
\begin{align}\label{SA.1}
\{x\in K: h(x)=0\} \subset \{x\in K: g(x) = 0\},
\end{align}
then there exists a strictly increasing, definable, concave function 
$\varphi \in C[0, \infty) \cap C^1(0, \infty)$ with $\varphi(0) = 0$ such that 
$$
g(x) \le \varphi(h(x)),  \quad \forall x \in K\cap \emph{dom}(h).
$$
\end{proposition}

\begin{proof}
For any $t \in \RR$ we set $h^{-1}(t) :=\{x\in K: h(x) =t\}$ and similarly for $g^{-1}(t)$. 
If $h^{-1}(0) = K$, then $g^{-1}(0)= K$ by (\ref{SA.1}) and thus the result holds 
trivially with $\psi(t) = t$. If $h^{-1}(0) = \emptyset$, then by the lower semi-continuity 
and nonnegativity of $h$ we can see that $h$ has a positive lower bound on $K$. 
Since $g$ is continuous on $K$ and $K$ is compact, we have 
$$
0\le M: = \sup_{x\in K} g(x) <\infty 
$$
and thus the result holds again with $\psi(t) = C t$ for some positive constant $C$. We therefore 
need only to consider the situation that $h^{-1}(0) \ne \emptyset$ and $h^{-1}(0) \ne K$. 

Since $K$ is a definable set and $h: K \to {\mathbb R}$ is a definable function, 
$h(K)$ is a definable set in $\RR$ and hence $h(K)$ is a finite union of points and 
intervals with $0 \in h(K)$. 

If $0$ is an isolated point in $h(K)$, then there is a constant $\ep>0$ such that 
$h(x) \ge \ep$ whenever $h(x) \ne 0$. Therefore for $x \in K$ with $h(x) <\ep$ 
we have $h(x) =0$ and hence $g(x)=0$ by the given condition (\ref{SA.1}); while 
for $x \in K$ with $h(x)\ge \ep$ we have 
$$
g(x) \le M +1  \le \frac{M+1}{\ep} h(x). 
$$
We thus obtain the desired result with $\psi(t) := (M+1) t/\ep$. 

In the following we therefore assume $0$ is not an isolated point in $h(K)$. By the 
structure of $h(K)$ as a definable set in $\RR$, we can find $T>0$ such that 
$[0, T] \subset h(K)$. Thus $h^{-1}(t) \ne \emptyset$ for each $t \in [0, T]$. Define 
$$
\mu(t):= \sup_{x\in h^{-1}(t)} g(x), \qquad t \in [0, T]
$$
Then $\mu$ is a definable function on $[0, T]$ and $\mu(0) =0$ by the given 
condition (\ref{SA.1}). 

We claim that $\mu$ is continuous at $t= 0$. If not, then there exist $\d>0$ and 
a sequence $\{t_k\}\subset (0, T]$ with $t_k \to 0$ as $k \to \infty$ such that 
$\mu(t_k) \ge \d$ for all $k$. By the definition of supremum, there is 
$x_k \in h^{-1}(t_k)$ such that 
$$
g(x_k) \ge \mu(t_k) - \d/2 \ge \d/2, \quad \forall k.  
$$
Since $\{x_k\}\subset K$ and $K$ is compact, by taking a subsequence if necessary,
we may assume $x_k \to \bar x$ as $k \to \infty$ for some $\bar x \in K$. By the 
continuity of $g$ we then have $g(\bar x) \ge \d/2>0$. By the lower semi-continuity 
of $h$ we also have 
$$
0 \le h(\bar x) \le \liminf_{k \to \infty} h(x_k) = \lim_{k\to \infty} t_k = 0
$$
which means $h(\bar x) =0$. This together with (\ref{SA.1}) shows $g(\bar x) =0$ 
which is a contradiction. 

Now we apply Lemma \ref{OM:lem1} to $\mu$  and use the continuity of $\mu$ at $t = 0$
to conclude the existence of $\ep\in (0, T)$ such that $\mu \in C[0, \ep] \cap C^1(0, \ep]$.
$\mu(0) =0$, and $\mu$ is either constant or strictly monotone on $[0, \ep]$. 

If $\mu$ is constant on $[0, \ep]$, then by using $\mu(0) =0$ we have $\mu(t) =0$ 
for all $t \in [0, \ep]$. Thus for $x \in K$ with $h(x) \le \ep$ we have $g(x) =0$; 
while for $x \in K$ with $h(x) > \ep$ we have $g(x) \le \frac{M+1}{\ep} h(x)$. 
Thus the desired result holds with $\psi(t) = (M+1)t/\ep$.

If $\mu$ is strictly monotone on $[0, \ep]$, by applying Lemma \ref{OM:lem1} to $\mu$
and $\mu'$ we may have $\mu(t) >0$, $\mu'(t)>0$ and $\mu''(t)$ has a constant 
sign on $(0, \ep]$ by shrinking $\ep$ if necessary. By the definition of $\mu$ we have 
for all $x \in K$ with $t:= h(x) \le \ep$ that
$$
g(x) \le \mu(t) = \mu(h(x)).
$$
In case $\mu'' \le 0$ on $(0, \ep]$, $\mu$ is a concave function on $[0, \ep]$. Therefore, 
by introducing the function 
\begin{align*}
\varphi(t):= \left\{\begin{array}{lll}
C_0 \mu(t), & 0 \le t \le \ep, \\
C_0 \left(\mu(\ep) + \mu'(\ep) (t - \ep)\right), & \ep < t < \infty
\end{array}\right.
\end{align*}
with $C_0:= \max\{1, (M+1)/\mu(\ep)\}$ and noting that $g(x) < M+1$ for $x \in K$, it 
is easy to see that $g(x) \le \varphi(h(x))$ for all $x \in K$ and 
$\varphi \in C[0, \infty)\cap C^1(0, \infty)$ is strictly increasing, definable 
and concave with $\varphi(0) = 0$. In case $\mu''>0$ on $(0, \ep]$, $\mu$ is a convex 
function on $[0, \ep]$. By using $\mu(0) =0$ we then have $\mu(t) \le \frac{\mu(\ep)}{\ep}t$ 
for $t \in [0, \ep]$. Define 
$$
\varphi(t) = \max\left\{\frac{\mu(\ep)}{\ep}, \frac{M+1}{\ep}\right\} t, \quad t \in [0, \infty).
$$
Again we can easily verify $g(x) \le \varphi(h(x))$ for $x \in K$. The proof is therefore complete.
\end{proof}

By virtue of Proposition \ref{SA.lem3}, we are now ready to prove a growth error bound 
for definable functions which in particular implies (\ref{OM:GEB}). 

\begin{theorem}\label{thm3}
Let $f: {\mathbb R}^n \to (-\infty, \infty]$ be a proper, lower semi-continuous, 
definable function and let $\gamma\in \RR$ be such that 
$$
S := \{x \in \RR^n: f(x) \le \gamma\} \ne \emptyset. 
$$ 
Then for any $\bar x \in \RR^n$ and $r>0$ there exists a strictly increasing,  
definable, concave function $\varphi\in C[0, \infty) \cap C^1(0, \infty)$ with 
$\varphi(0) = 0$ such that
\begin{align}\label{GEB2}
d(x, S) \le \varphi([f(x) - \gamma]_+)
\end{align}
for all $x \in B_r(\bar x)\cap \emph{dom}(f)$.  
\end{theorem}

\begin{proof}
Let 
$$
g(x) := d(x, S) \quad \mbox{ and } \quad h(x):= [f(x) - \gamma]_+. 
$$
Then $g$ is continuous and $h$ is lower semi-continuous on ${\mathbb R}^n$. 
By the definability of $f$ we can see that both $g$ and $h$ are definable. 
Moreover, $h(x) =0$ implies $g(x) =0$. Thus, we may apply Proposition \ref{SA.lem3}
to $g$ and $h$ on the compact definable set $\overline{B_r(\bar x)}$ 
to conclude the proof. 
\end{proof}

For semi-algebraic functions we can improve the result in Theorem \ref{thm3} by establishing 
H\"{o}lder growth error bounds. To see this, recall that a set in ${\mathbb R}^n$ is called 
semi-algebraic if it is a union of finite many sets of the form 
$$
\{x\in {\mathbb R}^n: f(x) =0 \mbox{ and } f_i(x) <0 \mbox{ for } i = 1, \cdots,p\},
$$
where $f, f_i: {\mathbb R}^n \to {\mathbb R}$ are real polynomials, and a function 
$f: {\mathbb R}^n \to (-\infty, \infty]$ is called semi-algebraic if its graph is a 
semi-algebraic set in ${\mathbb R}^{n+1}$, see \cite{C2002}. It is known that 
$\mathcal{SA}:= \{\mathcal{SA}_n\}$ is the smallest o-minimal structure on the real field 
$({\mathbb R}, +, \cdot)$, where $\mathcal{SA}_n$ denotes the collection of all semi-algebraic 
sets in ${\mathbb R}^n$. Thus, for a proper, lower semi-continuous, semi-algebraic function 
$f: {\mathbb R}^n \to (-\infty, \infty]$, we may use Theorem \ref{thm3} to conclude that 
(\ref{GEB2}) holds with a strictly increasing, semi-algebraic, concave function 
$\varphi\in C[0, \infty) \cap C^1(0, \infty)$ satisfying $\varphi(0) = 0$. Applying the 
Puiseux lemma (\cite{W2004}) to $\varphi$ we can conclude that 
$$
\varphi(t) \le C t^{\a}, \quad \forall t \ge 0
$$
for some constants $C>0$ and $\a \in (0, 1]$. Consequently we have the following result. 

\begin{corollary}\label{cor2}
Let $f: {\mathbb R}^n \to (-\infty, \infty]$ be a proper, lower semi-continuous, semi-algebraic function 
and let $\gamma\in {\mathbb R}$ be such that $S:=\{x\in {\mathbb R}^n: f(x) \le \gamma\} \ne \emptyset$. 
Then for any $\bar x \in {\mathbb R}^n$ and $r>0$ there exist constants $C>0$ and $\a \in (0, 1]$ 
such that 
$$
d(x, S) \le C [f(x) - \gamma]_+^\a, \quad \forall x \in B_r(\bar x). 
$$
\end{corollary}

\begin{remark}
For a given semi-algebraic function, Corollary \ref{cor2} establishes the H\"{o}lder growth 
error bound condition for some parameter $\a\in (0,1]$. Nevertheless, it does not specify the precise 
value of $\a$.  Determining this value presents a challenging task. Interested readers can 
find relevant insights in \cite{Li2013}, where further results on this aspect have been discussed.
\end{remark}

\section{\bf An application: adaptive heavy ball method}

Let $f: \RR^d \to \RR$ be a continuous differentiable convex function. We consider solving the convex 
minimization problem 
\begin{align}\label{conv.1}
\min_{x\in {\mathbb R}^d} f(x)
\end{align}
by the heavy ball method
\begin{align}\label{conv.2}
x_{k+1} = x_k - \a_k \nabla f(x_k) + \beta_k (x_k - x_{k-1}),
\end{align}
where $x_{-1} = x_0$ is the initial guess, $\a_k$ is the step-size and $\beta_k$ is the 
momentum coefficient. We assume that the gradient $\nabla f$ is $L$-Lipschitz continuous, i.e. 
\begin{align}\label{conv.3}
\|\nabla f(x) - \nabla f(y) \| \le L\|x-y\|, \quad \forall x, y \in \RR^d.
\end{align}
The performance of (\ref{conv.2}) depends crucially on the property of $f$ and the choices of 
the step-size $\a_k$ and the momentum coefficient $\beta_k$, and its analysis turns out to 
be very challenging. When $f$ is twice continuous differentiable, strongly convex and has 
Lipschitz continuous gradient, it has been demonstrated in \cite{P1964}, under suitable 
constant choices of $\a_k$ and $\beta_k$, that the iterative sequence $\{x_k\}$ enjoys a provable 
linear convergence faster than the gradient descent. Due to the recent development of machine 
learning and the appearance of large scale problems, the heavy ball method has gained renewed 
interest and much attention has been paid on understanding its convergence behavior 
(\cite{GFJ2015,N2018,O2018,OCBP2014,ZK1993}). Although many strategies have been proposed 
to select $\a_k$ and $\beta_k$, how to choose these parameters to achieve acceleration effect 
for general convex and non-convex problems remains an active research topic in optimization 
community. 

In this section we will consider the method (\ref{conv.2}) for solving convex problem (\ref{conv.1}) 
under the assumption that the minimum value 
$$
f_*:= \min_{x\in \RR^d} f(x) 
$$
is known and provide a strategy on choosing the step-size $\a_k$ and the momentum coefficient $\beta_k$ 
adaptively. We will then show that our proposed method, i.e. Algorithm \ref{alg:AHB} below, 
is well-defined and convergent. 

In order for the method (\ref{conv.2}) to have fast convergence, it is natural to choose $\a_k$ and 
$\beta_k$ at each iteration step such that $x_{k+1}$ is as close to a solution of (\ref{conv.1}) 
as possible. Let $\hat x$ denote any solution of (\ref{conv.1}). A naive idea is to directly minimize
$\|x_{k+1} - \hat x\|^2$ with respect to $\a_k$ and $\beta_k$. Unfortunately, directly minimizing 
this quantity becomes infeasible because it involves an unknown solution $\hat x$. As a compromise, 
we may first derive a suitable upper bound for this quantity and then take minimization steps. 

We will elucidate the idea on choosing $\a_k$ and $\beta_k$ below. For simplicity of 
exposition, we introduce the notation
$$
g_k := \nabla f(x_k) \quad \mbox{ and } \quad 
m_k := x_k - x_{k-1}
$$
for all integers $k\ge 0$. Then  (\ref{conv.2}) can be written as 
\begin{align}\label{conv.4}
x_{k+1} = x_k - \a_k g_k + \beta_k m_k. 
\end{align}
By the polarization identity, we can obtain
\begin{align*}
\|x_{k+1} - \hat x\|^2 
& = \|x_k - \hat x - \a_k g_k + \beta_k m_k\|^2 \\
& = \|x_k - \hat x\|^2 + \|\a_k g_k - \beta_k m_k\|^2 + 2 \l -\a_k g_k + \beta_k m_k, x_k - \hat x\r \displaybreak[0]\\
& = \|x_k - \hat x\|^2 + \a_k^2 \|g_k\|^2 - 2\a_k \beta_k \l g_k, m_k\r + \beta_k^2 \|m_k\|^2 \\
& \quad \, - 2 \a_k \l g_k, x_k - \hat x\r + 2 \beta_k \l m_k, x_k - \hat x\r.
\end{align*}
Since $f$ is convex and $\nabla f$ is $L$-Lipschitz continuous, we may use $\nabla f(\hat x) =0$ and 
the co-coercive property of $\nabla f$ to obtain  
$$
\l g_k, x_k - \hat x\r \ge f(x_k) - f(\hat x) + \frac{1}{2L} \|g_k\|^2.
$$
Therefore, by using $f(\hat x) = f_*$, we have  
\begin{align}\label{conv.5}
\|x_{k+1} - \hat x\|^2 
& \le \|x_k - \hat x\|^2 - \left(\frac{1}{L}-\a_k\right) \a_k \|g_k\|^2 
- 2\a_k \left(f(x_k) - f_*\right) \nonumber \\
& \quad \, - 2\a_k \beta_k \l g_k, m_k\r + \beta_k^2 \|m_k\|^2 
+ 2 \beta_k \l m_k, x_k - \hat x\r.
\end{align}
Note that when $\beta_k = 0$ for all $k$, the method (\ref{conv.2}) becomes the gradient descent 
method and extensive analysis has been done with the choices of the step-size $\a_k = \mu/L$
for some $0<\mu < 2$. For the method (\ref{conv.2}) we also choose $\a_k$ as such, i.e.  
\begin{align}\label{conv.6}
\a_k := \frac{1+\mu_0}{L} \quad \mbox{ with } 0 \le \mu_0 < 1. 
\end{align}
Then, by plugging the choice of $\a_k$ from (\ref{conv.6}) into (\ref{conv.5}) and using the inequality 
$\frac{1}{2L} \|g_k\|^2 \le f(x_k) - f_*$, we have 
\begin{align}\label{conv.7}
\|x_{k+1} -\hat x\|^2 
& \le \|x_k - \hat x\|^2 - 2(1-\mu_0)\a_k(f(x_k) - f_*) + \beta_k^2 \|m_k\|^2 \nonumber \\
& \quad \, - 2\a_k \beta_k \l g_k, m_k\r + 2 \beta_k \l m_k, x_k - \hat x\r.
\end{align}
We next consider the choice of $\beta_k$. A natural idea to obtain it is to minimize the right 
hand side of the above equation with respect to $\beta_k$. By minimizing the right hand side of 
(\ref{conv.6}) with respect to $\beta_k$ over the interval $[0, \beta]$ for a preassigned number 
$\beta \in (0, \infty]$ we can obtain 
$$
\beta_k = \min\left\{\max\left\{0, \frac{\a_k \l g_k, m_k\r - \l m_k, x_k - \hat x\r}{\|m_k\|^2}\right\}, \beta\right\}
$$
whenever $m_k \ne 0$. However, this formula for $\beta_k$ is not computable because 
it involves $\hat x$ which is unknown. We need to treat the term 
\begin{align}\label{conv.8}
\gamma_k := \l m_k, x_k - \hat x \r 
\end{align}
by using a suitable computable surrogate. Note that $\gamma_0 = 0$ and, for $k\ge 1$, 
\begin{align*}
\gamma_k & = \l m_k, x_k - x_{k-1} \r + \l m_k, x_{k-1} - \hat x\r \\
& = \|m_k\|^2 + \l -\a_{k-1} g_{k-1} + \beta_{k-1} m_{k-1}, x_{k-1} - \hat x\r \\
& = \|m_k\|^2 - \a_{k-1} \l g_{k-1}, x_{k-1} - \hat x \r 
+ \beta_{k-1} \gamma_{k-1}.
\end{align*}
By using again the co-coercivity of $\nabla f$ we have 
$$
\l g_{k-1}, x_{k-1} - \hat x\r \ge f(x_{k-1}) - f_* + \frac{1}{2L} \|g_{k-1}\|^2
$$ 
and thus 
\begin{align}\label{conv.9}
\gamma_k & \le  \|m_k\|^2 - \a_{k-1} \left(f(x_{k-1}) - f_* + \frac{1}{2L} \|g_{k-1}\|^2\right) 
+ \beta_{k-1} \gamma_{k-1}.
\end{align} 
This motivates us to introduce $\{\tilde \gamma_k\}$ such that 
\begin{align}\label{conv.10}
\tilde \gamma_k = \left\{\begin{array}{lll} 
0 & \mbox{ if } k =0, \\
\|m_k\|^2 - \a_{k-1} \left(f(x_{k-1}) - f_* + \frac{1}{2L} \|g_{k-1}\|^2\right) 
+ \beta_{k-1} \tilde \gamma_{k-1} & \mbox{ if } k \ge 1
\end{array}\right.
\end{align}
once $x_l$ for $0\le l \le k$ and $\a_l, \beta_l$ for $0\le l \le k-1$ are defined.  Consequently, it 
leads us to propose Algorithm \ref{alg:AHB} below. 

\begin{algorithm}[AHB: Adaptive heavy ball method]\label{alg:AHB}
{\it 
Take an initial guess $x_0$, and set $x_{-1} = x_0$. Pick $\beta\in (0, \infty]$ and $0\le \mu_0< 1$. 
For $n\ge 0$ do the following: 
		
\begin{enumerate} 

\item[\emph{(i)}] Calculate $g_k := \nabla f(x_k)$ and determine $\a_k$ according to (\ref{conv.6}); 
			
\item[\emph{(ii)}] Calculate $m_k := x_k - x_{k-1}$ and determine $\tilde \gamma_k$ by the formula (\ref{conv.10}); 
			
\item[\emph{(iii)}] Calculate $\beta_k$ by 
\begin{align*}
\beta_k = \left\{\begin{array}{lll}
\min\left\{\max\left\{0, \frac{\a_k \l g_k, m_k\r - \tilde \gamma_k}{\|m_k\|^2}\right\}, \beta\right\}
& \emph{ if } m_k \ne 0,\\
0 & \emph{ if } m_k =0; 
\end{array}\right.
\end{align*}
			
\item[\emph{(iv)}] Update $x_{k+1}$ by $x_{k+1} = x_k - \a_k g_k + \beta_k m_k$.
\end{enumerate}
}
\end{algorithm}

Note that, for the implementation of one step of Algorithm \ref{alg:AHB}, the most expensive part is 
the calculation of $g_k$ which is common for the gradient method; the computational load for other 
parts relating to $\tilde \gamma_k$ and $\beta_k$ can be negligible. Therefore, the computational 
complexity per iteration of Algorithm \ref{alg:AHB} is marginally higher than, but very close to, 
that of one step of the gradient method. 

\begin{lemma}\label{AHB.lem1}
Assume $\nabla f$ satisfies (\ref{conv.3}) and consider Algorithm \ref{alg:AHB}.
Define $\gamma_k$ by (\ref{conv.8}) for $k\ge 0$. Then $\gamma_k \le \tilde \gamma_k$ 
for all $k \ge 0$.  
\end{lemma}
	
\begin{proof}
Since $\gamma_0 = \tilde \gamma_0 =0$, the result is true for $k =0$. Assume next that 
$\gamma_k \le \tilde \gamma_k$ for some $k \ge 0$. By noting that $\beta_k \ge 0$, 
we may use (\ref{conv.9}) and the definition of $\tilde \gamma_{k+1}$ to obtain 
\begin{align*}
\gamma_{k+1}
& \le \|m_{k+1}\|^2 - \a_k \left(f(x_k) - f_* + \frac{1}{2L} \|g_k\|^2\right) + \beta_k \gamma_k \\
& \le \|m_{k+1}\|^2 - \a_k\left(f(x_k) - f_* + \frac{1}{2L} \|g_k\|^2\right) + \beta_k \tilde \gamma_k \\
& = \tilde \gamma_{k+1}. 
\end{align*}
By the induction principle, this shows the result. 
\end{proof}
	
\begin{lemma}\label{AHB.lem2}
Assume $\nabla f$ satisfies (\ref{conv.3}) and consider Algorithm \ref{alg:AHB}.  Then 
\begin{align*}
\|x_{k+1} - \hat x\|^2 \le \|x_k - \hat x\|^2 - c_0 (f(x_k) - f_*) 
\end{align*}
for all $k \ge 0$, where $c_0 := 2(1-\mu_0^2)/L > 0$ and $\hat x$ denotes any minimizer of (\ref{conv.1}).
\end{lemma}
	
\begin{proof}
According to (\ref{conv.7}) and Lemma \ref{AHB.lem1} we have for any solution $\hat x$ of (\ref{conv.1}) 
that
\begin{align*}
&\|x_{k+1} - \hat x\|^2 - \|x_k - \hat x\|^2 \\
& \le - 2(1-\mu_0)\a_k (f(x_k) - f_*) + \beta_k^2 \|m_k\|^2 + 2 \beta_k \gamma_k - 2 \a_k \beta_k \l g_k, m_k\r \\
& \le - 2(1-\mu_0)\a_k (f(x_k) - f_*) + \beta_k^2 \|m_k\|^2 + 2 \beta_k \tilde \gamma_k - 2 \a_k \beta_k \l g_k, m_k\r.
\end{align*}
Note that $\beta_k$ is the minimizer of the function $t \to h_k(t)$ over $[0, \beta]$, where 
$$
h_k(t) := - 2(1-\mu_0)\a_k(f(x_k) - f_*) + t^2 \|m_k\|^2 + 2 t \tilde \gamma_k  - 2 \a_k t \l g_k, m_k\r,
$$
we can conclude that 
\begin{align*}
\|x_{k+1} - \hat x\|^2 - \|x_k - \hat x\|^2 \le h_k(\beta_k) \le h_k(0) = - 2(1-\mu_0)\a_k (f(x_k) - f_*)
\end{align*} 
which shows the desired result. 
\end{proof}

Based on Lemma \ref{AHB.lem2} we will investigate the convergence of Algorithm \ref{alg:AHB}. 
In order to derive convergence rates under H\"{o}lder growth error bound conditions, we need 
the following calculus lemma. 

\begin{lemma}\label{lem_cal}
Let $\{\Delta_k\}$ be a sequence of nonnegative numbers satisfying
\begin{align}\label{5.8.6}
\Delta_{k+1} \le \Delta_k - C \Delta_k^\theta, \quad \forall k \ge 0, 
\end{align}
where $C>0$ and $\theta>1$ are constants. Then there is a constant $\tilde C>0$ such that
$$
\Delta_k \le \tilde C (1+ k)^{-\frac{1}{\theta-1}}, \quad \forall k \ge 0.
$$
\end{lemma}

\begin{proof}
Please refer to \cite[\S2.2, Lemma 6]{P1987}. 
\end{proof}

\begin{theorem}
Consider the minimization problem (\ref{conv.1}) in which $f$ is a continuous differentiable convex
function satisfying (\ref{conv.3}). Assume that the set $S$ of all solution of (\ref{conv.1}) is 
nonempty and that the minimum value $f_*$ of $f$ is known. Let $\{x_k\}$ be the sequence determined 
by Algorithm \ref{alg:AHB}. Then there exists a solution $x^\dag$ of (\ref{conv.1}) such that 
\begin{align*}
\|x_k - x^\dag\| \to 0 \quad \mbox{ as } k \to \infty. 
\end{align*}
If, in addition, there exist constants $\kappa>0$ and $\a \ge 2$ such that 
\begin{align}\label{AHB.geb}
\kappa [d(x, S)]^\a \le f(x) - f_*, \quad \forall x \in B_R(x^\dag),
\end{align}
where $R$ is a number such that $\|x_k - x^\dag\| < R$ for all $k \ge 0$, then there holds the following 
convergence rate:
\begin{enumerate}[leftmargin = 0.9cm]
\item[\emph{(i)}] If $\a = 2$, then $\{x_k\}$ converges to $x^\dag$ linearly, i.e. 
\begin{align}\label{AHB.rate1}
\|x_{k+1} - x^\dag\|^2 \le \left(1 - \frac{(1-\mu_0^2)\kappa}{4L} \right) \|x_k - x^\dag\|^2, \quad 
\forall k \ge 0.
\end{align}

\item[\emph{(ii)}] If $\a >2$, then there exists a positive constant $C$ independent of $k$ such that 
\begin{align}\label{AHB.rate2}
\|x_k - x^\dag\| \le C (k+1)^{-\frac{1}{\a-2}}, \quad \forall k \ge 0. 
\end{align}
\end{enumerate}
\end{theorem}

\begin{proof}
The convergence can be proved by a standard argument. According to Lemma \ref{AHB.lem2}, for any 
solution $\hat x$ of (\ref{conv.1}) the sequence $\{\|x_k - \hat x\|\}$ is monotonically decreasing 
and $f(x_k) - f_*\to 0$ as $k \to \infty$. Thus $\{x_k\}$ is bounded and hence it has a subsequence 
$\{x_{k_l}\}$ converging to a point $x^\dag$. By the continuity of $f$ we then have 
$$
f(x^\dag) = \lim_{l \to \infty} f(x_{k_l}) = f_*
$$
which implies that $x^\dag$ is a solution of (\ref{conv.1}). Thus $\{\|x_k - x^\dag\|\}$ is monotonically 
decreasing and $\|x_{k_l} - x^\dag\| \to 0$ as $l \to \infty$. Consequently $\|x_k - x^\dag\| \to 0$ as 
$k \to \infty$. 

Next we derive the convergence rates. By using Lemma \ref{AHB.lem2} and (\ref{AHB.geb}) we can obtain 
\begin{align*}
\|x_{k+1} - x^\dag\|^2 \le \|x_k - x^\dag\|^2 - c_0 \kappa [d(x_k, S)]^\a.
\end{align*}
Let $x\in S$ be any element. Since $\{\|x_k - x\|\}$ is monotonically decreasing, we have 
\begin{align*}
\|x_k - x^\dag\| \le \|x_k - x\| + \|x^\dag - x\| 
= \|x_k - x\| + \lim_{l \to \infty} \|x_l - x\| \le 2 \|x_k - x\|.
\end{align*}
Since $x \in S$ is arbitrary, this implies that $\|x_k - x^\dag\| \le 2 d(x_k, S)$. Therefore 
\begin{align*}
\|x_{k+1} - x^\dag\|^2 \le \|x_k - x^\dag\|^2 - 2^{-\a} c_0 \kappa \|x_k - x^\dag\|^\a.
\end{align*}
If $\a = 2$, we immediately obtain (\ref{AHB.rate1}). If $\a >2$, then we may use Lemma \ref{lem_cal}
to obtain (\ref{AHB.rate2}). 
\end{proof}

We conclude this section by providing a numerical experiment to test the performance of Algorithm 
\ref{alg:AHB}. We will compare our method with the following methods: 

\begin{enumerate}[leftmargin = 0.9cm]
\item[$\bullet$] Gradient method with constant step-size, i.e. 
$$
x_{k+1} = x_k - \a \nabla f(x_k)
$$
with $\a = \mu/L$ for some $\mu \in (0, 2)$. 

\item[$\bullet$] ALR-HB (heavy-ball method with adaptive learning rate) in \cite{WJZ2023}. This is the method 
(\ref{conv.2}) with $\beta_k \equiv \beta$ for some $\beta \in (0,1)$ and 
$$
\a_k = \frac{1}{2L} + \frac{f(x_k) - f_*}{\|g_k\|^2} + \beta \frac{\l g_k, x_k - x_{k-1}\r}{\|g_k\|^2}. 
$$
The nice numerical performance of ALR-HB has been demonstrated in \cite{WJZ2023}, the theoretical 
convergence guarantee however is not yet available. 

\item[$\bullet$] Nesterov's accelerated gradient method (\cite{AP2016,N1983}). This method takes the form 
\begin{align*}
z_k = x_k + \frac{k-1}{k+\nu} (x_k - x_{k-1}), \quad 
x_{k+1} = z_k - \a \nabla f(z_k)
\end{align*}
with $0< \a \le 1/L$ and $\nu \ge 2$. 
\end{enumerate}

\begin{example}
We consider applying Algorithm \ref{alg:AHB} to the standard 2D fan-beam tomography which consists in 
determining the density of cross sections of a human body by measuring the attenuation of X-rays as they 
propagate through the biological tissues. This imaging modality can be mathematically expressed as finding 
a compactly supported function from its line integrals, the so-called Radon transform (\cite{N2001}). We 
discretize the sought image on a $256\times 256$ pixel grid and identify it by a long vector in 
${\mathbb R}^N$ with $N = 256\times 256 = 65536$ by stacking all its columns. We consider the reconstruction 
from the tomographic data with $p=180$ projections and 367 X-ray lines per projection. By using the function 
\texttt{fanbeamtomo} from the MATLAB package AIR TOOLS \cite{HS2012} to discretize the problem, it leads to 
an ill-conditioned linear algebraic system $Ax=y$, where $A$ is a coefficient matrix with the size 
$66060\times 65536$. The true image $x^\dag$ in our experiment is the Shepp-Logan phantom. We calculate 
$y := Ax^\dag$ and then use it to reconstruct $x^\dag$ by solving the least square problem 
\begin{align*}
\min_{x \in \RR^N} \left\{f(x):= \frac{1}{2} \|A x - y\|^2\right\}.
\end{align*}
Clearly $f$ is a convex continuous differentiable function satisfying (\ref{conv.2}) with $L = \|A\|^2$
and its minimum value is $f_*=0$. In figure \ref{fig1} (a) we report the relative error $\|x_k - x^\dag\|/\|x^\dag\|$
of the computational results versus the iteration number $k$ for the above methods, all with the same 
initial guess $x_0 = 0$, where ``\texttt{AHB}" stands for our Algorithm \ref{alg:AHB} with $\beta = 1$ 
and $\mu_0 = 0.96$, ``\texttt{ALR-HB}" represents the ALR-HB method with $\beta = 0.96$, ``\texttt{Nesterov}"
stands for the Nesterov's accelerated gradient method with $\a = 1/L$ and $\nu = 3$, and ``\texttt{Gradient}"
stands for the gradient method with the constant step-size $\a = 1.96/L$. The plot demonstrates the obvious 
acceleration effect of ALH-HB, Nesterov, and our AHB over the gradient method. Furthermore, our AHB method 
produces results comparable to those obtained by ALR-HB and Nesterov, with even better accuracy as the 
iterations proceed. In Figure \ref{fig1} (b) we plot the $\beta_n$ values obtained by our AHB method, showing 
that our method consistently produces positive momentum coefficients, which promote the acceleration effect.  
\end{example}

\begin{figure}[htb!]
    \centering
\subfigure{
    \includegraphics[scale=0.47]{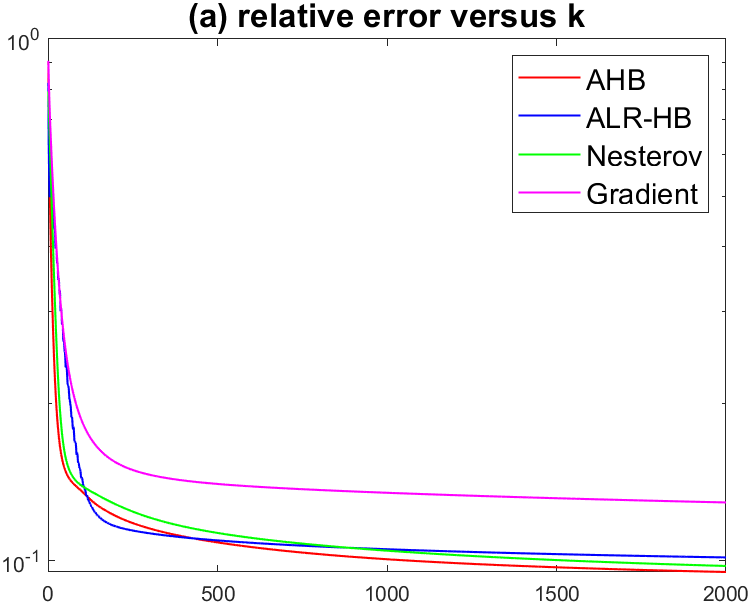}
}
\subfigure{
    \includegraphics[scale=0.47]{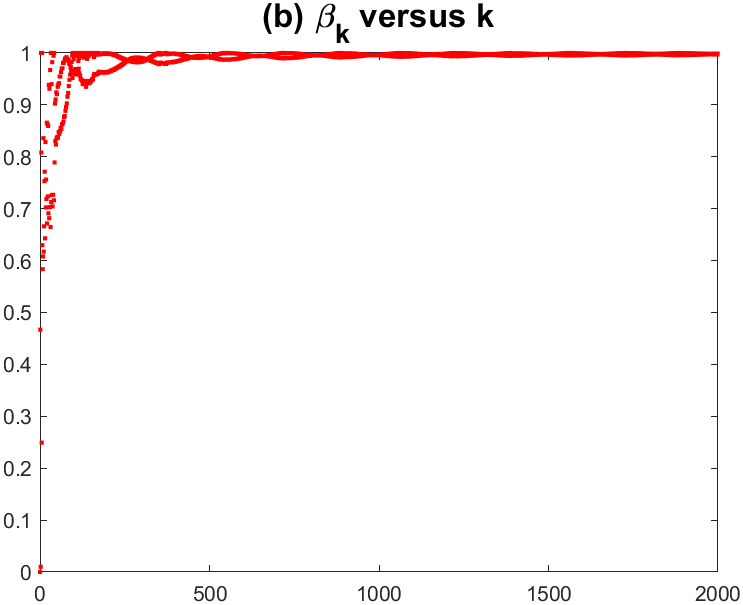}
}
\caption{}    
    \label{fig1}
\end{figure}

\medskip 
\noindent
{\bf Data Availability} Enquiries about data availability should be directed to the author.

\section*{\bf Declarations}
\noindent
{\bf Conflict of interest} The author declares that they have no conflict of interest.

\end{document}